\documentclass[reqno]{amsart}
\usepackage{amssymb}
\usepackage{amsmath}
\usepackage{color}
\usepackage{hyperref}
\hypersetup{colorlinks=true,linkcolor=blue,citecolor=red}

\newtheorem{theorem}{Theorem}[section]
\newtheorem{lemma}[theorem]{Lemma}

\theoremstyle{definition}

\numberwithin{equation}{section}

\newcommand{ \mint }{ {\int\hspace{-0.38cm}- }}

\newcommand{ \mr }{ \mathbb{R} }

\newcommand{\Norm}[1]{\left|\hspace{-0.3mm}\left| #1 \right|\hspace{-0.3mm}\right|}

\newcommand{\RN}[1]{%
  \textup{\uppercase\expandafter{\romannumeral#1}}%
}

\begin{document}

\title[Thin obstacle problem for the $p(x)$-Laplacian]{Regularity results of the thin obstacle problem for the $p(x)$-Laplacian}

\author{Sun-Sig Byun}
\address{Department of Mathematical Sciences and Research Institute of Mathematics,
Seoul National University, Seoul 08826, Korea}
\email{byun@snu.ac.kr}

\author{Ki-Ahm Lee}
\address{Department of Mathematical Sciences,
Seoul National University, Seoul 08826, Korea.
Center for Mathematical Challenges,
Korea Institute for Advanced Study, Seoul 02455, Korea}
\email{kiahm@snu.ac.kr}

\author{Jehan Oh}
\address{Department of Mathematical Sciences,
Seoul National University, Seoul 08826, Korea}
\email{ojhan0306@snu.ac.kr}

\author{Jinwan Park}
\address{Department of Mathematical Sciences,
Seoul National University, Seoul 08826, Korea}
\email{jinwann@snu.ac.kr}

\thanks{This work was supported by the National Research Foundation of Korea (NRF) grant funded by the Korea Government (NRF-2015R1A4A1041675).}


\date{\today.}


\maketitle

\begin{abstract}
We study thin obstacle problems involving the energy functional with $p(x)$-growth.
We prove higher integrability and H\"{o}lder regularity for the gradient of minimizers of the thin obstacle problems under the assumption that the variable exponent $p(x)$ is H\"{o}lder continuous.
\end{abstract}

\section{\bf Introduction}
\label{sec1}
In this paper we consider a \emph{thin obstacle problem for the $p(x)$-Laplacian}.
More precisely, we investigate a minimizer of the functional
\begin{equation}
\label{main_ftnl}
\mathcal{F}_{p(\cdot)}(v) := \int_{B^{+}_1} \frac{1}{p(x)} |Dv(x)|^{p(x)} \, dx
\end{equation}
over a convex admissible set 
\begin{equation}
\label{main_adset}
\mathcal{A} = \left\lbrace v \in W^{1,p(\cdot)}(B^{+}_1) : v=g \ \ \mathrm{on} \ \ (\partial B_1)^{+} \ \ \mathrm{and} \ \ v \geq 0 \ \ \mathrm{on} \ \ T_1 \right\rbrace,
\end{equation}
where $B^{+}_1 = B_1 \cap \{ x_n>0 \}$, $T_1 = B_1 \cap \{ x_n=0 \}$ with $n \geq 2$, and $g \in W^{1,p(\cdot)}(B^{+}_1)$.
Here, a variable exponent $p(\cdot) : \overline{B^{+}_1} \to (1,\infty)$ is assumed to be at least continuous, and satisfy
\begin{equation}
\label{bound_of_p(x)}
1 < \gamma_1 \leq p(x) \leq \gamma_2 < \infty
\end{equation}
for some constants $\gamma_1$ and $\gamma_2$.

For the case $p(x) \equiv 2$, the problem is called the \emph{boundary thin obstacle problem for the Laplacian} or the \emph{Signorini problem}.
This problem originates from optimal control of temperature \cite{Ath, Fre}, modelling of semipermeable membranes \cite{DL}, and financial mathematics \cite{ACS, Sil}.
The optimal regularity ($C^{1,\frac{1}{2}}$-regularity) for a minimizer of the problem for $n=2$ was shown by Richardson \cite{Ric}.
For a general dimesion $n \geq 2$, the $C^{1,\alpha}$-regularity for some $0 < \alpha \leq \frac{1}{2}$ was proved by Caffarelli \cite{Caf} and the optimal ($C^{1,\frac{1}{2}}$) result was achieved by Athanasopoulos and Caffarelli \cite{AC}.
For the case $p(x) \equiv p \in (1,\infty)$, the $C^{0,\alpha}$-regularity and the gradient estimates for a minimizer of the obstacle problem were established by B\"{o}gelein, F. Duzaar and Mingione \cite{BDM}, and the $C^{1,\alpha}$-regularity for a minimizer of the thin obstacle problem was obtained by Andersson and Mikayelyan \cite{AMik}.

The aim of this paper is to extend the $C^{1,\alpha}$-regularity obtained in \cite{AMik} to the thin obstacle problem for the $p(x)$-Laplacian.
In the process, we find a minimal regularity requirement on the variable exponent $p(x)$ to ensure the higher integrability, Theorem \ref{Higher integrability}, and the $C^{1,\alpha}$-regularity, Theorem \ref{MainThm_2}, for the thin obstacle problem (\ref{main_ftnl})-(\ref{main_adset}).

A main point in this paper is that the exponent $p(x)$ in the functional under consideration is not a constant function.
The functional (\ref{main_ftnl}) with $p(x)$-growth was first considered by Zhikov \cite{Zh1} in the context of homogenization, and in recent years there has been an increasing interest in this functional which provides a number of models arising in mathematical physics.
For instance, the functionals with $p(x)$-growth appear in the modelling of electro-rheological fluid \cite{RR1, Ru1, Ru2}, porous medium \cite{AS, LGCY}, image restoration \cite{CLR}, and fluid with temperature-dependent viscosity \cite{Zh2}.
Therefore, a great deal of work has been developed around variational problems with the $p(x)$-energy functional.
In particular, the higher integrability result for a minimizer of the $p(x)$-energy functional was obtained by Zhikov \cite{Zh1} (see also \cite{CM1}), and the $C^{1,\alpha}$-regularity was proved by Acerbi, Coscia and Mingione \cite{AMin, CM1}.
Here we establish these regularity results for a minimizer of the $p(x)$-energy functional with a thin obstacle.

We briefly introduce our approach to the proofs of our main results; Theorem \ref{Higher integrability} and Theorem \ref{MainThm_2}.
In order to get the higher integrability result, Theorem \ref{Higher integrability}, we shall derive a variational inequality of the thin obstacle problem and consider an associated problem whose solution enjoys the $C^{1,\alpha}$-regularity.
We then obtain the desired result by using Caccioppoli inequality, Poincar\'{e} inequality and Gehring lemma.
To prove Theorem \ref{MainThm_2}, we shall consider a minimizer of the thin obstacle problem for the $p$-Laplacian and derive a local estimate of the minimizer by comparison with an associated thin obstacle problem.

This paper is organized as follows.
In the next section, we present notation, function spaces and auxiliary lemmas.
Section \ref{sec3} is devoted to prove the higher integrability result.
In the last section, we finally prove the $C^{1,\alpha}$-regularity for a minimizer of the thin obstacle problem for the $p(x)$-Laplacian.

\section{\bf Preliminaries}
\label{sec2}

We start with introducing basic notation.

\begin{enumerate}
\item For a point $y \in \mr^n$ and for $r>0$, $B_r(y):=\{ x \in \mr^n : |x-y|<r \}$, $B^{+}_r(y) := B_r(y) \cap \{ x_n > 0 \}$, $(\partial B_r(y))^{+} := \partial B_r(y) \cap \{ x_n > 0 \}$, $T_r(y) := B_r(y) \cap \{ x_n = 0 \}$.
If the center is clear in the context, we shall omit denoting it as follows: $B_r \equiv B_r(y)$, $B^{+}_r \equiv B^{+}_r(y)$, $(\partial B_r)^{+} \equiv (\partial B_r(y))^{+}$, $T_r \equiv T_r(y)$.
\item For a function $f \in L_{\mathrm{loc}}^1(\mr^n)$ and a bounded open set $\Omega \subset \mr^n$, let $(f)_\Omega$ denote the integral average of $f$ in $\Omega$, that is,
\begin{equation*}
(f)_\Omega := \mint_\Omega f \, dx = \frac{1}{|\Omega|} \int_\Omega f \, dx.
\end{equation*}
\end{enumerate}

From now on, for the sake of convenience, we employ the letter $c$ to denote any universal constants which can be explicitly computed in terms of known quantities such as $n, \gamma_1, \gamma_2$.
Thus the exact value denoted by $c$ might be different from line to line.

\subsection{Function spaces}
\label{sec2_sub1}

Given a bounded domain $\Omega \subset \mr^n$ and a bounded measurable function $p(\cdot) : \Omega \subset \mr^n \rightarrow (1,\infty)$, \emph{the variable exponent Lebesgue space} $L^{p(\cdot)}(\Omega; \mr^N)$, $N \geq 1$, consists of all measurable functions $f : \Omega \rightarrow \mr^N$ such that
\begin{equation*}
\int_{\Omega} |f(x)|^{p(x)} dx < +\infty
\end{equation*}
with the following Luxemburg norm
\begin{equation*}
\Norm{f}_{L^{p(\cdot)}(\Omega; \mr^N)} := \inf \left\lbrace \lambda>0 : \int_{\Omega} \left| \frac{f(x)}{\lambda} \right|^{p(x)} dx \leq 1 \right\rbrace.
\end{equation*}

\emph{The variable exponent Sobolev space} $W^{1,p(\cdot)}(\Omega; \mr^N)$ is a collection of all measurable functions $f : \Omega \rightarrow \mr^N$ such that $f$ is weakly differentiable and its gradient $Df$ belongs to $L^{p(\cdot)}(\Omega; \mr^{Nn})$, that is
\begin{equation*}
W^{1,p(\cdot)}(\Omega; \mr^N) := \left\lbrace f \in L^{p(\cdot)}(\Omega; \mr^N) : Df \in L^{p(\cdot)}(\Omega; \mr^{Nn}) \right\rbrace,
\end{equation*}
equipped with the $W^{1,p(\cdot)}$-norm
\begin{equation*}
\Norm{f}_{W^{1,p(\cdot)}(\Omega; \mr^N)} := \Norm{f}_{L^{p(\cdot)}(\Omega; \mr^N)} + \Norm{Df}_{L^{p(\cdot)}(\Omega; \mr^{Nn})}.
\end{equation*}
We denote by $W_0^{1,p(\cdot)}(\Omega; \mr^N)$ to mean the closure of $C_0^{\infty}(\Omega; \mr^N)$ in $W^{1,p(\cdot)}(\Omega; \mr^N)$.
We notice that if $1 < \gamma_1 \leq p(\cdot) \leq \gamma_2 < \infty$ for some constants $\gamma_1$ and $\gamma_2$, then $L^{p(\cdot)}(\Omega; \mr^N)$, $W^{1,p(\cdot)}(\Omega; \mr^N)$ and $W_0^{1,p(\cdot)}(\Omega; \mr^N)$ are separable reflexive Banach spaces. For $N=1$, we simply write $L^{p(\cdot)}(\Omega)$, $W^{1,p(\cdot)}(\Omega)$ and $W_0^{1,p(\cdot)}(\Omega)$.
For further properties regarding variable exponent spaces, we refer to \cite{DHHR, DH, DR, DS, Ha, KR} and references therein.

We now present the Camapanato's spaces.
Let $p \geq 1$ and $\lambda \geq 0$.
We denote by $\mathcal{L}^{p,\lambda}(\Omega; \mr^N)$ to mean the space of functions $f \in L^p(\Omega; \mr^N)$ such that
\begin{equation*}
[f]_{p,\lambda} := \left\lbrace \sup_{\substack{x_0 \in \Omega \\ 0 < \rho < \mathrm{diam} \, \Omega}} \rho^{-\lambda} \int_{\Omega_{\rho}(x_0)} |f-(f)_{\Omega_{\rho}(x_0)}|^p \, dx \right\rbrace^{\frac{1}{p}} < +\infty,
\end{equation*}
where $\Omega_{\rho}(x_0) = \Omega \cap B_{\rho}(x_0)$.
We remark that the quantity $[f]_{p,\lambda}$ is a seminorm in $\mathcal{L}^{p,\lambda}$ and is equivalent to the quantity
\begin{equation*}
\left\lbrace \sup_{\substack{x_0 \in \Omega \\ 0 < \rho < \mathrm{diam} \, \Omega}} \rho^{-\lambda} \inf_{\xi \in \mr^N} \int_{\Omega_{\rho}(x_0)} |f-\xi|^p \, dx \right\rbrace^{\frac{1}{p}}.
\end{equation*}
Moreover, $\mathcal{L}^{p,\lambda}$ is a Banach space with the norm
\begin{equation*}
\Norm{f}_{p,\lambda} := \Norm{f}_{L^p} + [f]_{p,\lambda}.
\end{equation*}

The Campanato's spaces provide the following integral characterization of H\"{o}lder continuous functions.

\begin{lemma}\label{Campanato_embed}
\cite[Theorem 2.9]{Giu}
Let $\Omega$ be a bounded domain in $\mr^n$, and let $p \geq 1$ and $n < \lambda \leq n+p$.
Suppose that there exists a constant $A>0$ such that for every $x_0 \in \overline{\Omega}$ and for every $\rho \in (0,\mathrm{diam}\, \Omega)$, we have $|\Omega_{\rho}(x_0)| = |\Omega \cap B_{\rho}(x_0)| \geq A \rho^n$.
Then the space $\mathcal{L}^{p,\lambda}(\Omega; \mr^N)$ is isomorphic to the space $C^{0,\alpha}(\Omega; \mr^N)$ with $\alpha = \frac{\lambda-n}{p}$.
\end{lemma}

\subsection{Auxiliary lemmas}
\label{sec2_sub2}

We shall use the following Sobolev-Poincar\'{e} type inequality.

\begin{lemma}\label{So-Po ineq}
\cite{Eva}
Let $1 < \gamma_1 \leq p \leq \gamma_2 < \infty$ and $r>0$.
For any $f \in W^{1,p}(B_r^+)$, we have
\begin{align*}
\mint_{B_r^+} \left( \frac{|f-(f)_{B_r^+}|}{r} \right)^p dx & \leq c(n,p) \left( \mint_{B_r^+} |Df|^{\frac{np}{n+p}} \, dx \right)^{\frac{n+p}{n}} \\
& \leq c(n,\gamma_1,\gamma_2) \left( \mint_{B_r^+} |Df|^{\frac{np}{n+\gamma_1}} \, dx \right)^{\frac{n+\gamma_1}{n}}.
\end{align*}
\end{lemma}

We now state and prove a technical lemma.
For the standard technical lemma, we refer the reader to \cite[Lemma 2.1 of Chpater 3]{Gia} and \cite[Lemma 3.4]{HL}.

\begin{lemma}\label{technical lemma}
Let $\varphi$ be a non-negative and non-decreasing function on $[0,r_0]$.
Suppose that
\begin{equation}
\label{tech_ineq_1}
\varphi(\rho) \leq A \left\lbrace \left( \frac{\rho}{r} \right)^{\alpha_1} + \varepsilon \right\rbrace \varphi(2r) + Br^{\alpha_2}
\end{equation}
for all $0 < \rho < r \leq \frac{r_0}{2}$, where $A$, $B$, $\alpha_1$, $\alpha_2$ are non-negative constants with $\alpha_1 > \alpha_2$.
Then there exists $\varepsilon_0 = \varepsilon_0(A,\alpha_1,\alpha_2)>0$ such that if $0 \leq \varepsilon < \varepsilon_0$, we have
\begin{equation}
\label{tech_ineq_2}
\varphi(\rho) \leq c \left\lbrace \left( \frac{\rho}{r} \right)^{\alpha_2} \varphi(r) + B \rho^{\alpha_2} \right\rbrace
\end{equation}
for all $0 < \rho < r \leq r_0$, where $c=c(A,\alpha_1,\alpha_2)$ is a positive constant.
\end{lemma}

\begin{proof}
For $\kappa \in (0,\frac{1}{2})$ and $r \leq r_0$, we can rewrite (\ref{tech_ineq_1}) as
\begin{align*}
\varphi(\kappa r) & \leq A \left\lbrace (2\kappa)^{\alpha_1} + \varepsilon \right\rbrace \varphi(r) + B \left( \frac{r}{2} \right)^{\alpha_2} \\
& \leq (2\kappa)^{\alpha_1} A \left\lbrace 1 + \varepsilon \kappa^{-\alpha_1}  \right\rbrace \varphi(r) + B r^{\alpha_2}.
\end{align*}
We now choose $\kappa \in (0,\frac{1}{2})$ and $\varepsilon_0 >0$ in such a way that $2^{\alpha_1+1} \kappa^{\alpha_1} A \leq \kappa^{\alpha_3}$ with $\alpha_1 > \alpha_3 > \alpha_2$ and $\varepsilon_0 \kappa^{-\alpha_1} < 1$.
Then we get for every $r \leq r_0$,
\begin{equation*}
\varphi(\kappa r) \leq \kappa^{\alpha_3} \varphi(r) + B r^{\alpha_2}.
\end{equation*}
Therefore, for all integers $m \geq 0$, we have
\begin{align*}
\varphi(\kappa^{m+1} r) & \leq \kappa^{\alpha_3} \varphi(\kappa^m r) + B \kappa^{m\alpha_2} r^{\alpha_2} \\
& \leq \kappa^{(m+1)\alpha_3} \varphi(r) + B \kappa^{m\alpha_2} r^{\alpha_2} \sum_{j=0}^{m} \kappa^{j(\alpha_3-\alpha_2)} \\
& \leq c \kappa^{(m+1)\alpha_2} \left\lbrace \varphi(r) + Br^{\alpha_2} \right\rbrace
\end{align*}
for some constant $c=c(A,\alpha_1,\alpha_2)>1$.
Choosing $m$ such that $\kappa^{m+2}r < \rho \leq \kappa^{m+1}r$, we obtain (\ref{tech_ineq_2}) for all $\rho \in (0, \kappa r)$.
Since $\varphi$ is a non-decreasing function on $[0,r_0]$, we also discover that for $\rho \in (\kappa r, r)$,
\begin{equation*}
\varphi(\rho) \leq \varphi(r) = \left( \frac{1}{\kappa} \right)^{\alpha_2}  \kappa^{\alpha_2} \varphi(r) \leq \left( \frac{1}{\kappa} \right)^{\alpha_2} \left\lbrace \left( \frac{\rho}{r} \right)^{\alpha_2} \varphi(r) + B \rho^{\alpha_2} \right\rbrace,
\end{equation*}
which proves the lemma.
\end{proof}

\section{\bf Higher integrability of the gradient}
\label{sec3}

In this section, we establish the higher integrability for the gradient of a minimizer of the functional $\mathcal{F}_{p(\cdot)}$ in (\ref{main_ftnl}) over the admissible set $\mathcal{A}$ in (\ref{main_adset}).
We first present a variational inequality of the thin obstacle problem.

\begin{lemma}
\label{EL-lemma}
Let $u \in W^{1,p(\cdot)}(B^{+}_1)$ be a minimizer of the functional $\mathcal{F}_{p(\cdot)}$ over the admissible set $\mathcal{A}$ and let
\begin{equation}
\label{sub_adset}
\mathcal{A}_0 = \left\lbrace v \in W^{1,p(\cdot)}(B^{+}_1) : \ v=0 \ \ \mathrm{on} \ \ (\partial B_1)^{+} \ \ \mathrm{and} \ \ v \geq -u \ \ \mathrm{on} \ \ T_1 \right\rbrace.
\end{equation}
Then we have
\begin{equation}
\label{EL}
\int_{B^{+}_1} |Du|^{p(x)-2} Du \cdot Dv \, dx \geq 0, \quad \forall v \in \mathcal{A}_0.
\end{equation}
\end{lemma}

\begin{proof}
Let $v \in \mathcal{A}_0$.
Then we see at once that $u+tv=g$ on $(\partial B_1)^{+}$ and that $u+tv \geq (1-t)u \geq 0$ on $T_1$ for $0 \leq t \leq 1$.
Hence $u+tv \in \mathcal{A}$ for all $0 \leq t \leq 1$, where $\mathcal{A}$ is the admissible set in (\ref{main_adset}).
Since $u$ is a minimizer of $\mathcal{F}_{p(\cdot)}$, we have
\begin{equation}
\label{EL-pf1}
\left. \frac{d}{dt} \right|_{t=0} \mathcal{F}_{p(\cdot)}(u+tv) \geq 0.
\end{equation}
Observe that
\begin{align}
\nonumber \frac{d}{dt} \mathcal{F}_{p(\cdot)}(u+tv) & = \int_{B^{+}_1} \frac{d}{dt} \left[ \frac{1}{p(x)} (|Du+tDv|^2)^{\frac{p(x)}{2}} \right] dx \\
\nonumber & = \int_{B^{+}_1} (|Du+tDv|^2)^{\frac{p(x)}{2}-1} (Du+tDv) \cdot Dv \, dx \\
\label{EL-pf2} & = \int_{B^{+}_1} |Du+tDv|^{p(x)-2} (Du+tDv) \cdot Dv \, dx.
\end{align}
Combining (\ref{EL-pf1}) with (\ref{EL-pf2}), we obtain the desired conclusion (\ref{EL}).
\end{proof}

We now consider an associated problem whose solution has $C^{1,\theta}$-regularity for some $\theta \in (0,1)$, and show a comparison result.

\begin{lemma}
\label{Obs-Lemma}
Given a minimizer $u \in W^{1,p(\cdot)}(B^{+}_1)$ of $\mathcal{F}_{p(\cdot)}$ over the admissible set $\mathcal{A}$, let $w \in W^{1,p(\cdot)}(B^{+}_1)$ be the weak solution of
\begin{eqnarray}\begin{split}
\label{refer-prob}
\left\{
\begin{array}{cl}
- \mathrm{div} \left( |Dw|^{p(x)-2} Dw \right) = 0 \ & \mathrm{in} \ \ B^{+}_1, \\
\displaystyle w \equiv \inf_{(\partial B_{1})^{+}} u & \mathrm{on} \ \ (\partial B_1)^{+}, \\
w \equiv 0 & \mathrm{on} \ \ T_1. \\
\end{array}
\right.
\end{split}\end{eqnarray}
Then $w \in C^{1,\theta}(\overline{B^{+}_{\frac{3}{4}}})$ for some $\theta \in (0,1)$.
In addition, we have $u \geq w$ in $B^{+}_1$.
\end{lemma}

\begin{proof}
Let $\widetilde{w}$ denote the odd extension of $w$ from $B^{+}_1$ to $B_1$ as
\begin{eqnarray*}\begin{split}
\widetilde{w}(x) = \widetilde{w}(x_1, \cdots x_{n-1}, x_n) := \left\{
\begin{array}{cl}
w(x_1, \cdots x_{n-1}, x_n), & \mathrm{if} \ \ x_n \geq 0, \\
-w(x_1, \cdots x_{n-1}, -x_n), & \mathrm{if} \ \ x_n < 0,
\end{array}
\right.
\end{split}\end{eqnarray*}
and let $\widetilde{p}$ denote the even extension of $p$ from $B^{+}_1$ to $B_1$ as
\begin{eqnarray*}\begin{split}
\widetilde{p}(x) = \widetilde{p}(x_1, \cdots x_{n-1}, x_n) := \left\{
\begin{array}{cl}
p(x_1, \cdots x_{n-1}, x_n), & \mathrm{if} \ \ x_n \geq 0, \\
p(x_1, \cdots x_{n-1}, -x_n), & \mathrm{if} \ \ x_n < 0.
\end{array}
\right.
\end{split}\end{eqnarray*}
Then it follows from (\ref{refer-prob}) that $\widetilde{w} \in W^{1,\widetilde{p}(\cdot)}(B_1)$ is a weak solution of
\begin{equation*}
- \mathrm{div} \left( |D\widetilde{w}|^{\widetilde{p}(x)-2} D\widetilde{w} \right) = 0 \quad \mathrm{in} \ \ B_1.
\end{equation*}
Hence, we conclude that $\widetilde{w} \in C^{1,\theta}(\overline{B_{\frac{3}{4}}})$ for some $\theta \in (0,1)$ (see \cite{AMin}).
This yields that $w \in C^{1,\theta}(\overline{B^{+}_{\frac{3}{4}}})$.

We next prove that $u \geq w$ in $B^{+}_1$.
To see this, we start with an observation that for each $q \in (1,\infty)$, there exists a constant $c=c(q)>0$ such that
\begin{equation*}
(|\xi_1| + |\xi_2|)^{q-2} |\xi_1 - \xi_2|^2 \leq c \left( |\xi_1|^{q-2}\xi_1 - |\xi_2|^{q-2}\xi_2 \right) \cdot (\xi_1 - \xi_2)
\end{equation*}
for every $\xi_1, \xi_2 \in \mr^n$.
From this, we obtain that if $q \geq 2$, then
\begin{equation*}
|\xi_1 - \xi_2|^q \leq c \left( |\xi_1|^{q-2}\xi_1 - |\xi_2|^{q-2}\xi_2 \right) \cdot (\xi_1 - \xi_2)
\end{equation*}
for some $c=c(q)>0$.
If $1<q<2$, it follows from Young's inequality with $\varepsilon \in (0,1)$ that
\begin{align*}
|\xi_1 - \xi_2|^q & = (|\xi_1| + |\xi_2|)^{\frac{q(2-q)}{2}} (|\xi_1| + |\xi_2|)^{\frac{q(q-2)}{2}} |\xi_1 - \xi_2|^q \\
& \leq \varepsilon (|\xi_1| + |\xi_2|)^q + c \varepsilon^{-\frac{2-q}{q}} (|\xi_1| + |\xi_2|)^{q-2} |\xi_1 - \xi_2|^2 \\
& \leq \varepsilon (|\xi_1| + |\xi_2|)^q + c \varepsilon^{-1} (|\xi_1| + |\xi_2|)^{q-2} |\xi_1 - \xi_2|^2 \\
& \leq c \varepsilon (|\xi_1|^q + |\xi_2|^q) + c \varepsilon^{-1} \left( |\xi_1|^{q-2}\xi_1 - |\xi_2|^{q-2}\xi_2 \right) \cdot (\xi_1 - \xi_2),
\end{align*}
for some $c=c(q)>0$.
Therefore, if $1 < \gamma_1 \leq p(x) \leq \gamma_2 < \infty$, we have for any $\varepsilon \in (0,1)$,
\begin{align}
\nonumber |\xi_1 - \xi_2|^{p(x)} & \leq c \varepsilon (|\xi_1|^{p(x)} + |\xi_2|^{p(x)}) \\
\label{monotonicity} & \qquad + c \varepsilon^{-1} \left( |\xi_1|^{p(x)-2}\xi_1 - |\xi_2|^{p(x)-2}\xi_2 \right) \cdot (\xi_1 - \xi_2),
\end{align}
for some $c=c(\gamma_1,\gamma_2)>0$.

Now, by taking $v=(w-u)_{+} \in W_0^{1,p(\cdot)}(B^{+}_1) \subset \mathcal{A}_0$ as a test function of (\ref{refer-prob}) and (\ref{EL}), we get
\begin{equation*}
\int_{B^{+}_1 \cap \{ w>u \}} \left( |Dw|^{p(x)-2}Dw - |Du|^{p(x)-2}Du \right) \cdot (Dw - Du) \, dx \leq 0.
\end{equation*}
Therefore, it follows from (\ref{monotonicity}) that for any $\varepsilon \in (0,1)$,
\begin{equation*}
\int_{B^{+}_1 \cap \{ w>u \}} |Dw-Du|^{p(x)} \, dx \leq c \varepsilon \int_{B^{+}_1 \cap \{ w>u \}} \left( |Dw|^{p(x)} + |Du|^{p(x)} \right) dx.
\end{equation*}
Letting $\varepsilon \to 0$, we conclude that $w \leq u$ in $B^{+}_1$.
\end{proof}

From now on, we fix a minimizer $u \in W^{1,p(\cdot)}(B^{+}_1)$ of the thin obstacle problem (\ref{main_ftnl})-(\ref{main_adset}), and we take the weak solution $w \in W^{1,p(\cdot)}(B^{+}_1)$ to the problem (\ref{refer-prob}).
We write
\begin{equation}
\label{def_of_M}
M := \int_{B^{+}_1} \left( |Du|^{p(x)} + |Dw|^{p(x)} + 1 \right) dx + 1.
\end{equation}

Now, we are ready to prove the higher integrability of the gradient of $u$.

\begin{theorem}
[Higher integrability]\label{Higher integrability}
Suppose that there exists $\beta>0$ such that
\begin{equation}
|p(x)-p(y)| \leq [p(\cdot)]_{\beta} |x-y|^{\beta}, \quad \forall x, y \in \overline{B^{+}_1},
\end{equation}
for some constant $[p(\cdot)]_{\beta}>0$.
Then there exists $\sigma_0 = \sigma_0(n,\gamma_1,\gamma_2) \in (0,1)$ such that for any $B^{+}_{2r} = (B_{2r}(x_0))^{+} \subset B^{+}_{\frac{3}{4}}$ with $x_0 \in T_{\frac{1}{2}}$ and $r>0$ satisfying
\begin{equation}
\label{condition_of_r_1}
r \leq \min \left\lbrace \left( \frac{\beta}{8[p(\cdot)]_{\beta}} \right)^{\frac{2}{\beta}}, \ \frac{1}{4} \left( \frac{\gamma_1^2}{(2n+\gamma_1)[p(\cdot)]_{\beta}} \right)^{\frac{1}{\beta}}, \ \frac{1}{8M} \right\rbrace,
\end{equation}
there holds $Du \in L^{(1+\sigma_0)p(\cdot)}(B^{+}_r)$.
Moreover, for any $\sigma \in [0,\sigma_0]$ we have
\begin{equation}
\label{higher_integrability_estimate}
\mint_{B^{+}_r} |Du|^{(1+\sigma)p(x)} \, dx \leq c \left( \mint_{B^{+}_{2r}} |Du|^{p(x)} \, dx \right)^{1+\sigma} + c \mint_{B^{+}_{2r}} |Dw|^{(1+\sigma)p(x)} \, dx + c,
\end{equation}
for some constant $c=c(n,\gamma_1,\gamma_2)>1$.
\end{theorem}

\begin{proof}
Let $\eta \in C_0^{\infty}(B_{2r})$ be a cut-off function satisfying
\begin{equation}
\label{cut-off}
0 \leq \eta \leq 1, \quad \eta \equiv 1 \ \ \text{in} \ B_r \quad \text{and} \quad |D\eta| \leq \frac{c(n)}{r}.
\end{equation}
We observe from Lemma \ref{Obs-Lemma} that
\begin{equation*}
v := \eta^{\gamma_2} \left( w - (w)_{B^{+}_{2r}} - u + (u)_{B^{+}_{2r}} \right) \geq \eta^{\gamma_2}(w-u) = - \eta^{\gamma_2} u \geq -u \quad \mathrm{on} \ \ T_1.
\end{equation*}
Hence $v \in \mathcal{A}_0$, where $\mathcal{A}_0$ is the admissible set in (\ref{sub_adset}).
From Lemma \ref{EL-lemma}, we have
\begin{align*}
\int_{B^{+}_1} |Du|^{p(x)} \eta^{\gamma_2} \, dx & \leq \int_{B^{+}_1} \left( |Du|^{p(x)-2} Du \cdot Dw \right) \eta^{\gamma_2} \, dx \\
& \hspace{-0.6cm} + \gamma_2 \int_{B^{+}_1} \left( |Du|^{p(x)-2} Du \cdot D\eta \right) \eta^{\gamma_2-1} \left( w - (w)_{B^{+}_{2r}} - u + (u)_{B^{+}_{2r}} \right) dx.
\end{align*}
It follows from (\ref{cut-off}) and Young's inequality with $\varepsilon \in (0,1)$ that
\begin{align*}
\int_{B^{+}_{2r}} & |Du|^{p(x)} \eta^{\gamma_2} \, dx \\
& \leq \varepsilon \int_{B^{+}_{2r}} |Du|^{p(x)} \eta^{\gamma_2} \, dx + c(\varepsilon) \int_{B^{+}_{2r}} |Dw|^{p(x)} \eta^{\gamma_2} \, dx \\
& \quad + \varepsilon \int_{B^{+}_{2r}} |Du|^{p(x)} \eta^{\frac{(\gamma_2-1)p(x)}{p(x)-1}} \, dx + c(\varepsilon) \int_{B^{+}_{2r}} |w - (w)_{B^{+}_{2r}}|^{p(x)} |D\eta|^{p(x)} \, dx \\
& \quad + c(\varepsilon) \int_{B^{+}_{2r}} |u - (u)_{B^{+}_{2r}}|^{p(x)} |D\eta|^{p(x)} \, dx \\
& \leq 2\varepsilon \int_{B^{+}_{2r}} |Du|^{p(x)} \eta^{\gamma_2} \, dx + c(\varepsilon) \int_{B^{+}_{2r}} |Dw|^{p(x)} \eta^{\gamma_2} \, dx \\
& \quad + c(\varepsilon) \int_{B^{+}_{2r}} \left| \frac{w - (w)_{B^{+}_{2r}}}{r} \right|^{p(x)} dx + c(\varepsilon) \int_{B^{+}_{2r}} \left| \frac{u - (u)_{B^{+}_{2r}}}{r} \right|^{p(x)} dx,
\end{align*}
where we have used the fact that
\begin{equation*}
\frac{(\gamma_2-1)p(x)}{p(x)-1} = \frac{\gamma_2-1}{\gamma_2} \cdot \frac{p(x)}{p(x)-1} \cdot \gamma_2 \geq \gamma_2.
\end{equation*}
Taking $\varepsilon = \frac{1}{4}$ and using (\ref{cut-off}), we get
\begin{multline}
\mint_{B^{+}_r} |Du|^{p(x)} \, dx \leq c \mint_{B^{+}_{2r}} \left| \frac{u - (u)_{B^{+}_{2r}}}{r} \right|^{p(x)} dx \\
+ c \mint_{B^{+}_{2r}} \left| \frac{w - (w)_{B^{+}_{2r}}}{r} \right|^{p(x)} dx + c \mint_{B^{+}_{2r}} |Dw|^{p(x)} \, dx
\end{multline}
for some positive constant $c=c(n,\gamma_1,\gamma_2)$.

Let $\displaystyle p_1=\inf_{B^{+}_{2r}}p(\cdot)$ and $\displaystyle p_2=\sup_{B^{+}_{2r}}p(\cdot)$.
Since $p(\cdot) \in C^{0,\beta}(\overline{B^{+}_1})$, we have
$$p_2-p_1 \le [p(\cdot)]_\beta(2r)^\beta,$$
and then we obtain from the condition \eqref{condition_of_r_1} that
$$\frac{p_1}{p_2} \geq 1-\frac{[p(\cdot)]_\beta(2r)^\beta}{\gamma_1} \geq \frac{\frac{2n}{n+\gamma_1}}{\frac{n}{n+\gamma_1}+1}>\frac{n}{n+\gamma_1}.$$
Thus, by Lemma \ref{So-Po ineq} and H\"{o}lder's inequality for
\begin{equation*}
\frac{1}{\frac{n}{n+\gamma_1}}\frac{p_1}{p_2}\frac{\frac{n}{n+\gamma_1}+1}{2}>1 \quad \text{and} \quad \frac{1}{\frac{n}{n+\gamma_1}}\frac{p_1}{p_2}>1,
\end{equation*}
we have
\begin{align*}
\mint_{B^{+}_{2r}} \left| \frac{u - (u)_{B^{+}_{2r}}}{r} \right|^{p(x)} dx&\le \mint_{B^{+}_{2r}} \left| \frac{u - (u)_{B^{+}_{2r}}}{r} \right|^{p_2} dx+1\\
&\le c\left( \mint_{B^{+}_{2r}} \left|Du\right|^{\frac{np_2}{n+\gamma_1}} dx\right) ^{\frac{n+\gamma_1}{n}}+1\\
&\le c\left( \mint_{B^{+}_{2r}} \left|Du\right|^{p_1 \frac{\frac{n}{n+\gamma_1}+1}{2}} dx\right) ^{\frac{p_2}{p_1} \frac{2}{\frac{n}{n+\gamma_1}+1}}+1\\
&\le c\left( \mint_{B^{+}_{2r}} \left|Du\right|^{p(x) \frac{\frac{n}{n+\gamma_1}+1}{2}} dx\right) ^{\frac{p_2}{p_1} \frac{2}{\frac{n}{n+\gamma_1}+1}}+c
\end{align*}
and
\begin{align*}
\mint_{B^{+}_{2r}} \left| \frac{w - (w)_{B^{+}_{2r}}}{r} \right|^{p(x)} dx & \leq \mint_{B^{+}_{2r}} \left| \frac{w - (w)_{B^{+}_{2r}}}{r} \right|^{p_2} dx + 1 \\
& \leq c \left( \mint_{B^{+}_{2r}} \left|Dw\right|^{\frac{np_2}{n+\gamma_1}} dx\right) ^{\frac{n+\gamma_1}{n}} + 1 \\
& \leq c \left( \mint_{B^{+}_{2r}} \left|Dw\right|^{p_1} dx\right) ^{\frac{p_2}{p_1}} + 1 \\
& \leq c\left( \mint_{B^{+}_{2r}} \left|Dw\right|^{p(x)} dx\right) ^{\frac{p_2}{p_1}} + c.
\end{align*}

On the other hand, the condition \eqref{condition_of_r_1} implies that
\begin{equation*}
(p_2-p_1)\log\left(\frac{1}{r}\right)\le [p(\cdot)]_\beta(2r)^\beta\frac{2}{\beta}\left(\frac{1}{r}\right)^{\frac{\beta}{2}} \le 1.
\end{equation*}
This yields
\begin{equation}\label{du}
\left( \mint_{B_{2r}^+}|Du|^{p(x)} \, dx \right)^{p_2-p_1}\le c\left(\frac{M}{r^n}\right)^{p_2-p_1}\le c\left(\frac{1}{r^{n+1}}\right)^{p_2-p_1}\le c
\end{equation}
and
\begin{equation}\label{dw}
\left( \mint_{B_{2r}^+}|Dw|^{p(x)} \, dx \right)^{p_2-p_1}\le c.
\end{equation}
Hence, we obtain from H\"{o}lder's inequality that
\begin{align*}
\left(\mint_{B^{+}_{2r}} |Du|^{p(x)\frac{\frac{n}{n+\gamma_1}+1}{2}} \, dx \right)^{\frac{p_2}{p_1} \frac{2}{\frac{n}{n+\gamma_1}+1}} & \leq \left( \mint_{B_{2r}^+}|Du|^{p(x)} \, dx \right)^{\frac{p_2}{p_1}} \\
& = \left( \mint_{B_{2r}^+}|Du|^{p(x)} \,dx \right)^{\frac{p_2-p_1}{p_1}} \mint_{B_{2r}^+}|Du|^{p(x)} \, dx \\
& \leq c \mint_{B_{2r}^+}|Du|^{p(x)} \, dx
\end{align*}
and that
\begin{align*}
\left( \mint_{B_{2r}^+}|Dw|^{p(x)} \, dx \right)^{\frac{p_2}{p_1}} & = \left( \mint_{B_{2r}^+}|Dw|^{p(x)} \,dx \right)^{\frac{p_2-p_1}{p_1}} \mint_{B_{2r}^+}|Dw|^{p(x)} \, dx \\
& \leq c \mint_{B_{2r}^+}|Dw|^{p(x)} \, dx.
\end{align*}

Consequently, we get
\begin{equation*}
\mint_{B^{+}_r}|Du|^{p(x)} \, dx \leq c \left(\mint_{B^{+}_{2r}} |Du|^{p(x)\frac{\frac{n}{n+\gamma_1}+1}{2}} \, dx \right)^\frac{2}{\frac{n}{n+\gamma_1}+1} + c \mint_{B^{+}_{2r}} |Dw|^{p(x)} \, dx + c.
\end{equation*}
Since $Dw \in L^{\infty}(B^{+}_{\frac{3}{4}})$, we have the desired inequality \eqref{higher_integrability_estimate} by using Gehring's lemma (see \cite[Theorem 6.6]{Giu}).
\end{proof}

\section{\bf H\"{o}lder continuity of the gradient}
\label{sec4}

Under the same assumptions and conclusions as in Theorem \ref{Higher integrability}, we further investigate a finer regularity for the problem (\ref{main_ftnl})-(\ref{main_adset}).
Especially, we will prove the H\"{o}lder continuity for the gradient of a minimizer of the $p(x)$-energy functional (\ref{main_ftnl}) over the admissible set (\ref{main_adset}).

We recall that $B^{+}_{2r} = (B_{2r}(x_0))^{+} \subset B^{+}_{\frac{3}{4}}$ with $x_0 \in T_{\frac{1}{2}}$, where $r \in (0,\frac{1}{8})$ satisfies (\ref{condition_of_r_1}).
We set
\begin{equation*}
p_1 := \inf_{B^{+}_r} p(\cdot), \quad p_2 := \sup_{B^{+}_r} p(\cdot) \quad \text{and} \quad \Norm{Dw}_{\infty} := \Norm{Dw}_{L^{\infty}(B^{+}_{\frac{3}{4}})}.
\end{equation*}
We further assume on $r$ for which
\begin{equation}
\label{condition_of_sigma_0}
p_2 - p_1 \leq [p(\cdot)]_{\beta} (2r)^{\beta} \leq \frac{\sigma_0}{2} < \frac{1}{2}.
\end{equation}
Then we get
\begin{align}
\nonumber p_2 = p(x) + p_2 - p(x) & \leq p(x) + p_2 - p_1 \\
\label{ineq_p2} & \leq p(x) \left( 1 + p_2 - p_1 \right) \leq p(x) \left( 1 + \frac{\sigma_0}{2} \right).
\end{align}
According to (\ref{higher_integrability_estimate}) with $\sigma = \frac{\sigma_0}{2}$, we have $Du \in L^{p_2}(B^{+}_r)$.
Moreover, it follows from (\ref{higher_integrability_estimate}), (\ref{du}) and (\ref{ineq_p2}) that
\begin{align*}
\mint_{B^{+}_r} |Du|^{p_2} \, dx & \leq \mint_{B^{+}_r} |Du|^{p(x)(1+p_2-p_1)} \, dx + 1 \\
&\le c \left( \mint_{B^{+}_{2r}} |Du|^{p(x)} \, dx \right)^{1+p_2-p_1} + c \mint_{B^{+}_{2r}} |Dw|^{(1+p_2-p_1)p(x)} \, dx + c \\
& \leq c \mint_{B^{+}_{2r}} |Du|^{p(x)} \, dx + c \left( \Norm{Dw}_{\infty}^{\left( 1 + \frac{1}{2} \right) \gamma_2} + 1 \right) + c \\
& \leq c \mint_{B^{+}_{2r}} |Du|^{p(x)} \, dx + c \Norm{Dw}_{\infty}^{\frac{3\gamma_2}{2}} + c.
\end{align*}

We now consider a minimizer $u_0 \in W^{1,p_2}(B^{+}_r)$ of the functional
\begin{equation*}
\mathcal{F}_{p_2}(v) = \frac{1}{p_2} \int_{B^{+}_r} |Dv|^{p_2} \, dx
\end{equation*}
over the convex admissible set 
\begin{equation*}
\mathcal{B} = \left\lbrace v \in W^{1,p_2}(B^{+}_r) : v=u \ \ \mathrm{on} \ (\partial B_r)^{+} \ \ \mathrm{and} \ \ v \geq 0 \ \ \mathrm{on} \ T_r \right\rbrace.
\end{equation*}
We apply Lemma \ref{EL-lemma} when $p(x)$, $g$, $u$, $B^{+}_1$ are replaced by $p_2$, $u$, $u_0$, $B^{+}_r$, respectively, to discover that
\begin{equation}
\label{reference_EL}
\int_{B^{+}_r} |Du_0|^{p_2-2} Du_0 \cdot Dv \, dx \geq 0, \quad \forall v \in \mathcal{B}_0,
\end{equation}
where
\begin{equation}
\label{reference_sub_adset}
\mathcal{B}_0 = \left\lbrace v \in W^{1,p_2}(B^{+}_r) : v=0 \ \ \mathrm{on} \ \ (\partial B_r)^{+} \ \ \mathrm{and} \ \ v \geq -u_0 \ \ \mathrm{on} \ \ T_r \right\rbrace.
\end{equation}
Since $u\in W^{1,p_2}(B_r^+)$, we see that $u \in \mathcal{B}$.
Also, by the definition of $u_0$, we have
\begin{equation}
\label{DugeDu0}
\int_{B_r^+} |Du_0|^{p_2} \, dx \leq \int_{B_r^+} |Du|^{p_2} \,dx.
\end{equation}

Then we obtain the following comparison estimates.

\begin{lemma}
[Comparison estimates]
\label{Comparion_est_lemma}
With (\ref{condition_of_r_1}) and (\ref{condition_of_sigma_0}), we further assume that
\begin{equation}\label{condition_of_sigma_1}
r \leq \frac{1}{[p(\cdot)]_{\beta}} \left( \frac{\sigma_1}{4} \right)^{\frac{1}{\beta}}, \quad \mathrm{where} \ \ \sigma_1 := \min \left\lbrace \frac{\beta}{4n}, \ \sigma_0 \right\rbrace.
\end{equation}
Then we have
\begin{equation}\label{Du-Du_0 p_2}
\int_{B^{+}_r} |Du-Du_0|^{p_2} \, dx \leq cr^{\frac{\beta}{4}} \left( M^{\sigma_1} \int_{B^{+}_{2r}} |Du|^{p_2} \, dx + r^n \right),
\end{equation}
where $M$ is given as (\ref{def_of_M}), for some constant $c=c(n,\gamma_1,\gamma_2,[p(\cdot)]_{\beta}, \Norm{Dw}_{\infty})>1$.
\end{lemma}

\begin{proof}
Let us first observe that $u-u_0 \in \mathcal{B}_0$, where $\mathcal{B}_0$ is the admissible set in (\ref{reference_sub_adset}).
Moreover, by putting $u_0=u$ in $B^{+}_1 \setminus B^{+}_r$, we see that $u_0-u \in \mathcal{A}_0$, where $\mathcal{A}_0$ is the admissible set in (\ref{sub_adset}).
Therefore, we obtain from (\ref{EL}) and (\ref{reference_EL}) that
\begin{equation}
\label{comp1}
\int_{B^{+}_r} |Du_0|^{p_2-2} Du_0 \cdot (Du-Du_0) \, dx \geq 0
\end{equation}
and
\begin{equation}
\label{comp2}
\int_{B^{+}_r} |Du|^{p(x)-2} Du \cdot (Du_0-Du) \, dx \geq 0.
\end{equation}
Combining (\ref{comp1}) and (\ref{comp2}) yields
\begin{align*}
\RN{1} & := \int_{B^{+}_r} \left( |Du|^{p_2-2} Du - |Du_0|^{p_2-2} Du_0 \right) \cdot (Du-Du_0) \, dx \\
& \leq \int_{B^{+}_r} |Du|^{p_2-2} Du \cdot (Du-Du_0) \, dx \\
& \leq \int_{B^{+}_r} \left( |Du|^{p_2-2} Du - |Du|^{p(x)-2} Du \right) \cdot (Du-Du_0) \, dx =: \RN{2}.
\end{align*}
By \eqref{monotonicity} and \eqref{DugeDu0}, we get
\begin{equation}\label{Du-Du_0}
\int_{B_r^+}|Du-Du_0|^{p_2}dx \le c\varepsilon \int_{B_r^+} |Du|^{p_2} + c \varepsilon^{-1} \RN{1}
\end{equation}
for any $\varepsilon\in (0,1)$.

From the mean value theorem for the map $t \to |Du|^{t(p_2-p(x))}$, for $x\in B_r$, there exists $t_x \in (0,1)$ such that
\begin{multline}\label{DuDu}
\left||Du|^{p_2-2}Du-|Du|^{p(x)-2}Du\right| = \left|\left(|Du|^{p_2-p(x)}-1\right)|Du|^{p(x)-2}Du\right|\\
\le \left(p_2-p(x)\right)|Du|^{t_x(p_2-p(x))} \left|\log |Du| \right| |Du|^{p(x)-1}.
\end{multline}
By using $t^{\gamma_1-1
}|\log t|\le c(\gamma_1)$ for $0<t\le 1$, and $\log t\le c(\sigma) t^\sigma$ for $t>1$ with $\sigma>0$, 
we have 
$$|Du|^{t_x(p_2-p(x))} \left|\log |Du| \right| |Du|^{p(x)-1} \le c(\gamma_1),$$
for $|Du|\le 1$, while
$$|Du|^{t_x(p_2-p(x))} \left|\log |Du| \right| |Du|^{p(x)-1} \le c(\sigma) |Du|^\sigma|Du|^{p_2-1},$$
for $|Du| > 1.$
Thus, we get
\begin{multline}\label{aa}
\left(p_2-p(x)\right)|Du|^{t_x(p_2-p(x))} \left|\log |Du| \right| |Du|^{p(x)-1} \\
\le c \left(p_2-p_1\right)\left(|Du|^{\sigma_2}|Du|^{p_2-1}+1 \right),
\end{multline}
where $\sigma_2:=\frac{(\gamma_1-1)\sigma_1}{2\gamma_1}$.

By the definition of $\RN{2}$, \eqref{DuDu}, \eqref{aa}, \eqref{DugeDu0} and \eqref{condition_of_sigma_0}, we have
\begin{align*}
|\RN{2}|&\le c\left(p_2-p_1\right)\int_{B_r^+} \left( |Du-Du_0|+|Du|^{\sigma_2}|Du|^{p_2 -1}|Du-Du_0| \right) dx\\
&\le cr^\beta \int_{B_r^+} \left( \frac{|Du-Du_0|^{p_2}}{p_2}+\frac{|Du|^{\sigma_2\frac{p_2}{p_2 -1}}|Du|^{p_2}}{\frac{p_2}{p_2-1}}+1 \right) dx\\
&\le cr^\beta \int_{B_r^+} \left( |Du|^{p_2}+|Du|^{\sigma_2\frac{p_2}{p_2-1}}|Du|^{p_2} +1 \right) dx\\
&\le cr^\beta \int_{B_r^+} \left( |Du|^{p(x)+[p(\cdot)]_\beta(2r)^\beta+ \sigma_2\frac{\gamma_1}{\gamma_1-1}}+1 \right) dx\\
&\le cr^\beta \int_{B_r^+} \left( |Du|^{(1+\sigma_1)p(x)}+1 \right) dx,
\end{align*}
for some constant $c=c(n,\gamma_1,\gamma_2,[p(\cdot)]_{\beta})>1$. Moreover, by using Theorem \ref{Higher integrability}, we have
\begin{align*}
|\RN{2}|&\le cr^\beta \left\{ r^n \left( \mint_{B^{+}_{2r}} |Du|^{p(x)} \, dx \right)^{1+\sigma_1} +  \int_{B^{+}_{2r}} \left( |Dw|^{(1+\sigma_1)p(x)} +1 \right) dx  \right\}\\
&\le cr^\beta \left\{ r^{-n\sigma_1} M^{\sigma_1} \int_{B^{+}_{2r}} |Du|^{p_2} \, dx + r^n \right\}, \\
\end{align*}
for some constant $c=c(n,\gamma_1,\gamma_2,[p(\cdot)]_{\beta},\Norm{Dw}_{\infty})>1$.

Therefore, we obtain from \eqref{Du-Du_0}, $\RN{1} \le \RN{2}$ and the estimate for $\RN{2}$ that
\begin{align*}
\int_{B^{+}_r} |Du-Du_0|^{p_2} \, dx \leq & c\varepsilon \int_{B_r^+} |Du|^{p_2} \, dx \\
& \quad + c \varepsilon^{-1}r^\beta \left ( r^{-n\sigma_1}M^{\sigma_1} \int_{B^{+}_{2r}} |Du|^{p_2} \, dx + r^n \right) \\
\le &cr^{\frac{\beta}{4}} \left( M^{\sigma_1} \int_{B^{+}_{2r}} |Du|^{p_2} \, dx + r^n \right), 
\end{align*}
by taking $\varepsilon=r^{\frac{\beta-n\sigma_1}{2}}$ and using the fact that $\frac{\beta-n\sigma_1}{2} \geq \frac{\beta}{4}$.
\end{proof}

We now provide the results of H\"{o}lder regularity for the gradient of a minimizer of the thin obstacle problem for the $p$-Laplacian.

\begin{lemma}
Let $0 < r \leq 1$ and let $1 < \gamma_1 \leq p \leq \gamma_2 < \infty$ be fixed.
For a minimizer $u_0 \in W^{1,p}(B^{+}_r)$ of the functional
\begin{equation}
v \mapsto \mathcal{F}_p(v) = \frac{1}{p} \int_{B^{+}_r} |Dv|^p \, dx
\end{equation}
over the admissible set 
\begin{equation*}
\mathcal{A} = \left\lbrace v \in W^{1,p}(B^{+}_r) : v=g \ \ \mathrm{on} \ (\partial B_r)^{+} \ \ \mathrm{and} \ \ v \geq 0 \ \ \mathrm{on} \ T_r \right\rbrace
\end{equation*}
with $g \in W^{1,p}(B^{+}_r)$, then $u_0 \in C^{1,\alpha_0}(\overline{B^{+}_{\frac{r}{8}}})$ for some $\alpha_0 = \alpha_0(n,\gamma_1,\gamma_2) \in (0,1)$.
Moreover, there exists a constant $c=c(n,\gamma_1,\gamma_2)>1$ such that for any $0 < \rho < \frac{r}{8}$,
\begin{equation}
\label{aux_lem_r1}
\mint_{B^{+}_{\rho}} |Du_0-(Du_0)_{B^{+}_{\rho}}|^p \, dx \leq c \left( \frac{\rho}{r} \right)^{\alpha_0} \mint_{B^{+}_r} |Du_0|^p \, dx
\end{equation}
and
\begin{equation}\label{aux_lem_r2}
\mint_{B^{+}_{\rho}} |Du_0|^p \, dx \leq c \mint_{B^{+}_r} |Du_0|^p \, dx.
\end{equation}
\end{lemma}

\begin{proof}
We first observe that it is sufficient to prove the lemma only for the case $r=1$ by scaling.
Indeed, if we set $\displaystyle \widehat{u_0}(x) := \frac{1}{r}	u_0(rx)$ and $\displaystyle \widehat{g}(x) := \frac{1}{r} g(rx)$ for $x \in B^{+}_1$, then $D\widehat{u_0}(x) = Du_0(rx)$ for all $x \in B^{+}_1$, and hence $\widehat{u_0}$ is a minimizer of the functional
\begin{equation*}
v \mapsto \frac{1}{p} \int_{B^{+}_1} |Dv|^p \, dx
\end{equation*}
over the admissible set
\begin{equation*}
\widehat{\mathcal{A}} = \left\lbrace v \in W^{1,p}(B^{+}_1) : v=\widehat{g} \ \ \mathrm{on} \ (\partial B_1)^{+} \ \ \mathrm{and} \ \ v \geq 0 \ \ \mathrm{on} \ T_1 \right\rbrace.
\end{equation*}
Then we obtain that $\widehat{u_0} \in C^{1,\alpha_0}(\overline{B^{+}_{\frac{1}{8}}})$ for some $\alpha_0 = \alpha_0(n,\gamma_1,\gamma_2) \in (0,1)$ and that for any $0 < \rho < \frac{1}{8}$,
\begin{equation*}
\mint_{B^{+}_{\rho}} |D\widehat{u_0}-(D\widehat{u_0})_{B^{+}_{\rho}}|^p \, dx \leq c \rho^{\alpha_0} \mint_{B^{+}_1} |D\widehat{u_0}|^p \, dx
\end{equation*}
and
\begin{equation*}
\mint_{B^{+}_{\rho}} |D\widehat{u_0}|^p \, dx \leq c \mint_{B^{+}_1} |D\widehat{u_0}|^p \, dx,
\end{equation*}
where $c=c(n,\gamma_1,\gamma_2)>1$.
After scaling back, we conclude that $u_0 \in C^{1,\alpha_0}(\overline{B^{+}_{\frac{r}{8}}})$ and that for any $0 < \rho < \frac{r}{8}$,
\begin{equation*}
\mint_{B^{+}_{\rho}} |Du_0-(Du_0)_{B^{+}_{\rho}}|^p \, dx \leq c \rho^{\alpha_0} \mint_{B^{+}_r} |Du_0|^p \, dx \leq c \left( \frac{\rho}{r} \right)^{\alpha_0} \mint_{B^{+}_r} |Du_0|^p \, dx
\end{equation*}
and
\begin{equation*}
\mint_{B^{+}_{\rho}} |Du_0|^p \, dx \leq c \mint_{B^{+}_r} |Du_0|^p \, dx.
\end{equation*}

We now prove the lemma for the case $r=1$.
From the proof of Theorem 4.3 in \cite{AMik}, we can obtain that 
\begin{equation}
\label{aux_lem_pf_0}
\sup_{B^{+}_{\rho}} |Du_0| \leq c \left( \sup_{B^{+}_{\frac{1}{8}}} |D'u_0| \right) \rho^{\alpha}, \quad \forall \rho \in \left( 0, \frac{1}{8} \right),
\end{equation}
for some $\alpha = \alpha(n,p) \in (0,1)$ and $c=c(n,p)>1$, where $D'u_0 = (D_1 u_0, \cdots, D_{n-1} u_0)$.
Then it follows from a classical renormalization argument (see \cite{AMik}) that
\begin{equation}
\label{aux_lem_pf_1}
|Du_0(x_1)-Du_0(x_2)| \leq c \left( \sup_{B^{+}_{\frac{1}{4}}} |D'u_0| \right) |x_1-x_2|^{\alpha}, \quad \forall x_1, x_2 \in B^{+}_{\frac{1}{8}}.
\end{equation}
On the other hand, putting $v_0=|Du_0|^p$, we can deduce from \cite{AF, DSV} that for any $\eta \in C_0^{\infty}(B^{+}_{\frac{1}{2}})$ with $\eta \geq 0$,
\begin{equation*}
\int_{B^{+}_{\frac{1}{2}}} a_{ij}(x) D_i v_0 D_j \eta \, dx \leq -c \int_{B^{+}_{\frac{1}{2}}} \left| D (|Du|^{\frac{p-2}{2}}Du) \right|^2 \eta \, dx \leq 0,
\end{equation*}
for some positive constant $c=c(n,p)$, where
\begin{equation*}
a_{ij}(x) := \delta_{ij} + (p-2)\frac{D_i u(x) D_j u(x)}{|Du(x)|^2}.
\end{equation*}
Since $1 < \gamma_1 \leq p \leq \gamma_2 < \infty$, the matrix $A(x) = \left( a_{ij}(x) \right)$ is bounded and uniformly elliptic.
Hence, $v_0$ is a subsolution to
\begin{equation*}
- \mathrm{div} \left( A(x) Dv_0 \right) \leq 0 \quad \text{in} \ B^{+}_{\frac{1}{2}}.
\end{equation*}
Let $\widetilde{v_0}$ be the even extension of $v_0$ from $B^{+}_{\frac{1}{2}}$ to $B_{\frac{1}{2}}$, and let $\widetilde{a_{ij}}(x)$ be an extension of $a_{ij}(x)$ from $B^{+}_{\frac{1}{2}}$ to $B_{\frac{1}{2}}$ such that
\begin{eqnarray*}\begin{split}
\left\{
\begin{array}{ll}
\widetilde{a_{ij}}(x', -x_n) = a_{ij}(x', x_n) & \mathrm{for} \ \ 1 \leq i < n, 1 \leq j < n, \\
\widetilde{a_{in}}(x', -x_n) = -a_{in}(x', x_n) & \mathrm{for} \ \ 1 \leq i < n, \\
\widetilde{a_{nj}}(x', -x_n) = -a_{nj}(x', x_n) & \mathrm{for} \ \ 1 \leq j < n, \\
\widetilde{a_{nn}}(x', -x_n) = a_{nn}(x', x_n), \\
\end{array}
\right.
\end{split}\end{eqnarray*}
for $x=(x',x_n) \in B^{+}_{\frac{1}{2}}$.
Then we see that $\widetilde{v_0}$ is a non-negative function and that the matrix $\widetilde{A}(x) = \left( \widetilde{a_{ij}}(x) \right)$ is bounded and uniformly elliptic.
Moreover, $\widetilde{v_0}$ is a subsolution to
\begin{equation*}
- \mathrm{div} \left( \widetilde{A}(x) D\widetilde{v_0} \right) \leq 0 \quad \text{in} \ B_{\frac{1}{2}}.
\end{equation*}
We then use Moser iteration technique (see \cite{AF, DSV}) to obtain that $\widetilde{v_0} \in L^{\infty}(B_{\frac{1}{4}})$ with the estimate
\begin{equation*}
\sup_{B_{\frac{1}{4}}} \widetilde{v_0} \leq c \int_{B_{\frac{1}{2}}} \widetilde{v_0} \, dx,
\end{equation*}
and hence
\begin{equation}
\label{aux_lem_pf_2}
\sup_{B^{+}_{\frac{1}{4}}} |D'u_0|^p \leq \sup_{B^{+}_{\frac{1}{4}}} |Du_0|^p \leq c \int_{B^{+}_{\frac{1}{2}}} |Du_0|^p \, dx \leq c \int_{B^{+}_1} |Du_0|^p \, dx.
\end{equation}
Combining (\ref{aux_lem_pf_1}) with (\ref{aux_lem_pf_2}) gives
\begin{equation}
\label{aux_lem_pf_3}
|Du_0(x_1)-Du_0(x_2)| \leq c \left( \int_{B^{+}_1} |Du_0|^p \, dx \right)^{\frac{1}{p}} |x_1-x_2|^{\alpha}
\end{equation}
for all $x_1, x_2 \in B^{+}_{\frac{1}{8}}$.
We then find that for any $x_1 \in B^{+}_{\rho}$ with $0 < \rho < \frac{1}{8}$,
\begin{align*}
|Du_0(x_1)-(Du_0)_{B^{+}_{\rho}}| & = \left| \mint_{B^{+}_{\rho}} \left( Du_0(x_1)-Du_0(x) \right) dx \right| \\
& \leq \mint_{B^{+}_{\rho}} |Du_0(x_1)-Du_0(x)| \, dx \\
& \leq c \left( \int_{B^{+}_1} |Du_0|^p \, dx \right)^{\frac{1}{p}} \mint_{B^{+}_{\rho}} |x_1-x|^{\alpha} \, dx \\
& \leq c \left( \int_{B^{+}_1} |Du_0|^p \, dx \right)^{\frac{1}{p}} \rho^{\alpha}.
\end{align*}
Therefore, we conclude that for any $0 < \rho < \frac{1}{8}$,
\begin{equation*}
\mint_{B^{+}_{\rho}} |Du_0-(Du_0)_{B^{+}_{\rho}}|^p \, dx \leq c \rho^{p\alpha} \mint_{B^{+}_1} |Du_0|^p \, dx \leq c \rho^{\alpha_0} \mint_{B^{+}_1} |Du_0|^p \, dx
\end{equation*}
for some $\alpha_0 = \alpha_0(n,\gamma_1,\gamma_2) \in (0,1)$, where we have used the fact that $1 < \gamma_1 \leq p \leq \gamma_2 < \infty$.
This proves (\ref{aux_lem_r1}).
Furthermore, it follows from (\ref{aux_lem_pf_2}) that for any $0 < \rho < \frac{1}{8}$,
\begin{equation*}
\mint_{B^{+}_{\rho}} |Du_0|^p \, dx \leq \sup_{B^{+}_{\rho}} |Du_0|^p \leq \sup_{B^{+}_{\frac{1}{8}}} |Du_0|^p \leq c \mint_{B^{+}_1} |Du_0|^p \, dx,
\end{equation*}
which yields (\ref{aux_lem_r2}).
\end{proof}

Using the previous lemma and the comparison estimate in Lemma \ref{Comparion_est_lemma}, we have the following lemma.
We recall our assumptions (\ref{condition_of_r_1}), (\ref{condition_of_sigma_0}) and (\ref{condition_of_sigma_1}) on $r$.

\begin{lemma}
For any $\tau \in (0,n)$, there exists a constant $\delta \in (0,1)$ depending on $n, \gamma_1,\gamma_2,[p(\cdot)]_{\beta},\Norm{Dw}_{\infty},\tau$ such that if $r \leq \delta M^{-\frac{4\sigma}{\delta}}$ with  then we have
\begin{equation}
\label{decay_estimate}
\mint_{B^{+}_{\rho}} |Du|^{p_2} \, dx \leq c \rho^{-\tau} \left( \mint_{B^{+}_r} |Du|^{p_2} \, dx + 1 \right),
\end{equation}
whenever $\rho \in (0,r)$, for some constant $c=c(n,\gamma_1,\gamma_2,[p(\cdot)]_{\beta},\Norm{Dw}_{\infty},\tau)>1$.
\end{lemma}

\begin{proof}
By \eqref{DugeDu0}, \eqref{Du-Du_0 p_2} and \eqref{aux_lem_r2}, we have
\begin{align*}
\int_{B^{+}_{\rho}} |Du|^{p_2} \, dx &\le c \left( \int_{B^{+}_\rho} |Du-Du_0|^{p_2} \, dx+\int_{B^{+}_\rho} |Du_0|^{p_2} \, dx  \right)\\
&\le cr^{\frac{\beta}{4}} \left( M^{\sigma_1} \int_{B^{+}_{2r}} |Du|^{p_2} \, dx + r^n \right)+\left( \frac{\rho}{r} \right)^n\int_{B^{+}_r} |Du_0|^{p_2} \, dx\\
&\le c_1 \left\{ r^{\frac{\beta}{4}}  M^{\sigma_1}+\left( \frac{\rho}{r} \right)^n \right\}\int_{B^{+}_{2r}} |Du|^{p_2} \, dx+ c_2 r^{n-\tau},
\end{align*}
for some constants $c_1,c_2>1$ depending only on $n,\gamma_1,\gamma_2,[p(\cdot)]_{\beta}$ and $\Norm{Dw}_{\infty}$.
We set $\varphi(s) = \int_{B^{+}_{s}} |Du|^{p_2} \, dx$.
Let $\varepsilon_0$ be the positive number given in Lemma \ref{technical lemma} with $(A,B, \alpha_1, \alpha_2)=(c_1,c_2,n,n-\tau)$.
We then take $\delta>0$ such that
$$ r^{\frac{\beta}{4}}  M^{\sigma_1}\le \delta^{\frac{\beta}{4}} \le \varepsilon_0.$$
Applying Lemma \ref{technical lemma}, we conclude that
\begin{equation*}
\int_{B^{+}_{\rho}} |Du|^{p_2} \, dx \le c \left(\frac{\rho}{r}\right)^{n-\tau}\int_{B^{+}_{r}} |Du|^{p_2} \, dx+c\rho^{n-\tau}
\end{equation*}
whenever $\rho \in (0,r)$, and hence
\begin{equation*}
\mint_{B^{+}_{\rho}} |Du|^{p_2} \, dx \leq c\rho^{-\tau}\left( r^{\tau} \mint_{B^{+}_{r}} |Du|^{p_2} \, dx+1 \right) \leq c\rho^{-\tau}\left(\mint_{B^{+}_{r}} |Du|^{p_2} \, dx+1 \right),
\end{equation*}
which is the desired inequality.
\end{proof}

We finally state and prove the main theorem.

\begin{theorem}
[H\"{o}lder continuity of the gradient] \label{MainThm_2}
Let $u \in W^{1,p(\cdot)}(B^{+}_1)$ be a minimizer of the functional $\mathcal{F}_{p(\cdot)}$ over the admissible set $\mathcal{A}$.
Then there exists $\alpha = \alpha(n,\beta,\gamma_1,\gamma_2) \in (0,1)$ such that $u \in C^{1,\alpha}(B^{+}_{\frac{1}{2}})$.
\end{theorem}

\begin{proof}
Since $u \in W^{1,p(\cdot)}(B^{+}_1)$ satisfies the Euler-Lagrange equation
\begin{equation*}
- \mathrm{div} \left( |Du|^{p(x)-2} Du \right) = 0 \quad \text{in} \ B_r(x_0) \Subset B^{+}_1,
\end{equation*}
the gradient of $u$ is H\"{o}lder continuous in the interior of $B^{+}_1$ (see for instance \cite{AMin}).
Hence we only focus on the boundary case $x_0 \in T_{\frac{1}{2}}$.
Let $r_0$ satisfy \eqref{condition_of_r_1}, \eqref{condition_of_sigma_0} and \eqref{condition_of_sigma_1} for $r=r_0$. We also assume that
$$r_0^{\frac{\beta}{8}}M^{\sigma_1}\le 1,$$
where $\sigma_1$ and $M$ are given as (\ref{condition_of_sigma_1}) and (\ref{def_of_M}), respectively.
Suppose $0< \rho < \frac{r}{8} < \frac{r_0}{16}$ and set
$$p_+ := \sup_{B^{+}_{r_0}} p(\cdot), \quad p_- := \inf_{B^{+}_{r_0}} p(\cdot) \quad \text{ and } \quad p_2 := \sup_{B^{+}_{r}} p(\cdot).$$
It follows from the triangle inequality and H\"{o}lder's inequality that
\begin{align*}
\mint_{B^{+}_{\rho}} |Du-(Du)_{B^{+}_{\rho}}|^{p_2} \, dx & \leq c \mint_{B^{+}_{\rho}} |Du-(Du_0)_{B^{+}_{\rho}}|^{p_2} \, dx + c |(Du_0)_{B^{+}_{\rho}}-(Du)_{B^{+}_{\rho}}|^{p_2} \\
& = c \mint_{B^{+}_{\rho}} |Du-(Du_0)_{B^{+}_{\rho}}|^{p_2} \, dx + c \left| \mint_{B^{+}_{\rho}} \left( Du \, - (Du_0)_{B^{+}_{\rho}} \right) dx \right|^{p_2} \\
& \leq c \mint_{B^{+}_{\rho}} |Du-(Du_0)_{B^{+}_{\rho}}|^{p_2} \, dx \\
& \leq c \mint_{B^{+}_{\rho}} |Du-Du_0|^{p_2} \, dx + c \mint_{B^{+}_{\rho}} |Du_0-(Du_0)_{B^{+}_{\rho}}|^{p_2} \, dx,
\end{align*}
and hence
\begin{align*}
\int_{B^{+}_{\rho}} |Du-(Du)_{B^{+}_{\rho}}|^{p_2} \, dx & \leq c \int_{B^{+}_{\rho}} |Du-Du_0|^{p_2} \, dx + c \int_{B^{+}_{\rho}} |Du_0-(Du_0)_{B^{+}_{\rho}}|^{p_2} \, dx \\
& =: \RN{1} + \RN{2}.
\end{align*}
By (\ref{Du-Du_0 p_2}), we estimate $\RN{1}$ as follows:
\begin{align*}
\RN{1} \leq c \int_{B^{+}_r} |Du-Du_0|^{p_2} \, dx & \leq cr^{\frac{\beta}{4} + n} \left( M^{\sigma_1} \mint_{B^{+}_{2r}} |Du|^{p_2} \, dx + 1 \right) \\
& \leq cr^{\frac{\beta}{4} + n} \left( \mint_{B^{+}_{2r}} |Du|^{p_+} \, dx + 1 \right) \\
& \leq cr^{\frac{\beta}{4} + n - \tau} \left( \mint_{B^{+}_{r_0}} |Du|^{p_+} \, dx + 1 \right) \\
\end{align*}
for any $\tau \in (0,n)$, where we have used (\ref{decay_estimate}) with $(\rho, r, p_2)$ replaced by $(2r,r_0,p_+)$ for the last inequality.
To estimate $\RN{2}$, we apply (\ref{aux_lem_r1}) with $p=p_2$, (\ref{DugeDu0}) and (\ref{decay_estimate}) with $(\rho, r, p_2)$ replaced by $(r,r_0,p_+)$.
Then
\begin{align*}
\RN{2} \leq c \rho^n \left( \frac{\rho}{r} \right)^{\alpha_0} \mint_{B^{+}_r} |Du_0|^{p_2} \, dx & \leq c \rho^n \left( \frac{\rho}{r} \right)^{\alpha_0} \mint_{B^{+}_r} |Du|^{p_2} \, dx \\
& \leq c \rho^n \left( \frac{\rho}{r} \right)^{\alpha_0} \left( \mint_{B^{+}_r} |Du|^{p_+} \, dx + 1 \right) \\
& \leq c \rho^n \left( \frac{\rho}{r} \right)^{\alpha_0} r^{-\tau} \left( \mint_{B^{+}_{r_0}} |Du|^{p_+} \, dx + 1 \right)
\end{align*}
for any $\tau \in (0,n)$.
Combining these estimates, we have
\begin{equation*}
\int_{B^{+}_{\rho}} |Du-(Du)_{B^{+}_{\rho}}|^{p_2} \, dx \leq c \left( r^{\beta_0 + n - \tau} + \rho^{n+\alpha_0} r^{-\alpha_0 - \tau} \right) \left( \mint_{B^{+}_{r_0}} |Du|^{p_+} \, dx + 1 \right),
\end{equation*}
where $\beta_0 := \frac{\beta}{4}$.
We now choose $\rho=\frac{1}{8} r^{1+\frac{\beta_0}{n+\alpha_0}}$ to discover that
\begin{equation*}
\int_{B^{+}_{\rho}} |Du-(Du)_{B^{+}_{\rho}}|^{p_2} \, dx \leq c \rho^{\frac{(n+\alpha_0)(n + \beta_0 - \tau)}{n + \alpha_0 + \beta_0}} \left( \mint_{B^{+}_{r_0}} |Du|^{p_+} \, dx + 1 \right).
\end{equation*}
Finally, we select $\tau = \frac{\alpha_0 \beta_0}{2(n+\alpha_0)} \in (0,n)$, thereby deciding $\delta \in (0,1)$ and $r_0 \in (0,\frac{1}{8})$, to obtain
\begin{equation*}
\int_{B^{+}_{\rho}} |Du-(Du)_{B^{+}_{\rho}}|^{p_2} \, dx \leq c \rho^{n+\frac{\alpha_0 \beta_0}{2(n+\alpha_0+\beta_0)}} \left( \mint_{B^{+}_{r_0}} |Du|^{p_+} \, dx + 1 \right).
\end{equation*}
This yields
\begin{align*}
\left( \mint_{B^{+}_{\rho}} |Du-(Du)_{B^{+}_{\rho}}|^{\gamma_1} \, dx \right)^{\frac{1}{\gamma_1}} & \leq \left( \mint_{B^{+}_{\rho}} |Du-(Du)_{B^{+}_{\rho}}|^{p_2} \, dx \right)^{\frac{1}{p_2}} \\
& \leq c \rho^{\frac{\alpha_0 \beta_0}{2(n+\alpha_0+\beta_0)p_2}} \left( \mint_{B^{+}_{r_0}} |Du|^{p_+} \, dx + 1 \right)^{\frac{1}{p_2}} \\
& \leq c \rho^{\frac{\alpha_0 \beta_0}{2(n+\alpha_0+\beta_0)\gamma_2}} \left( \mint_{B^{+}_{r_0}} |Du|^{p_+} \, dx + 1 \right)^{\frac{1}{p_-}}
\end{align*}
for any $B^{+}_{\rho} = B^{+}_{\rho}(y) \subset B^{+}_{\widetilde{r_0}}(x_0)$, where $\widetilde{r_0} := \left( \frac{r_0}{16} \right)^{1+\frac{\beta_0}{n+\alpha_0}}$.
Then we conclude from Lemma \ref{Campanato_embed} that $Du \in C^{0,\alpha}(B^{+}_{\widetilde{r_0}}(x_0); \mr^n)$ with $\alpha=\frac{\alpha_0 \beta_0}{2(n+\alpha_0+\beta_0)\gamma_2}>0$.
\end{proof}

\bibliographystyle{amsplain}

\end{document}